\documentclass[a4paper,12pt,reqno,oneside]{amsart} % amsart loads amsmath, amsthm and amsfonts

\usepackage{setspace} 
\usepackage{typearea} % European margins
\usepackage{amscd,amssymb,verbatim} % amssymb loads amsfonts

\usepackage{xr-hyper}
\usepackage{cite}
\usepackage[colorlinks,final,bookmarksnumbered,bookmarks]{hyperref}
\usepackage{amsrefs}

\usepackage{ifthen}
\newcommand{\arxiv}[2][]{\ifthenelse{\equal{#1}{}}
{\href{http://arxiv.org/abs/#2}{\tt arXiv:#2}}
{\href{http://arxiv.org/abs/math/#2}{\tt arXiv:math.#1/#2}}}

\theoremstyle{plain}
\newtheorem{theorem}{Theorem}
\newtheorem{lemma}[theorem]{Lemma}

\theoremstyle{definition}
\newtheorem{remark}[theorem]{Remark}

\def\x{\times}
\def\but{\setminus}
\def\emb{\hookrightarrow}
\def\eps{\varepsilon}
\def\phi{\varphi}

\renewcommand{\:}{\colon}
\def\N{\Bbb{N}}
\def\R{\Bbb{R}}
\def\Z{\Bbb{Z}}

\begin{document}
\title{A note on O. Frolkina's paper ``Pairwise disjoint Moebius bands in space''}
\author{Sergey A. Melikhov}
\address{Steklov Mathematical Institute of Russian Academy of Sciences, Moscow, Russia}
\email{melikhov@mi-ras.ru}
%\date{\today}

%\begin{abstract} We note a simple algebraic proof for Frolkina's result that $\R^3$ does not contain
%uncountably many copies of the M\"obius band, and for a similar result in higher dimensions.  
%\end{abstract}
\maketitle

\begin{theorem} \label{t1} If $M$ is a non-orientable triangulable $(n-1)$-manifold (possibly with boundary), then
$\R^n$ does not contain uncountably many pairwise disjoint copies of $M$.
\end{theorem}

The case $n=3$ is due to Frolkina \cite{Fr}; her proof involves replacing the given copies of $M$ by tame ones. 
She also proves that $\R^n$ does not contain uncountably many pairwise disjoint {\it tame} copies of $M$.
The proof of Theorem \ref{t1} is an exercise in elementary algebraic topology, which I found myself doing as 
I was asked to referee Frolkina's paper. 

\begin{remark}
It is well-known that if $X$ is a compactum such that $\R^n$ contains uncountably many pairwise disjoint copies of $X$,
then $X\x\N^+$ embeds in $\R^n$, where $\N^+$ is the one-point compactification of the countable discrete space $\N$.
Indeed, the space of maps $C^0(X,\R^n)$ is separable, so any its uncountable subset $S$ contains a convergent sequence $(f_i)$ 
whose limit $f_\infty$ also belongs to $S$.
\end{remark}

\begin{lemma} \label{l1} If $X$ is a compact polyhedron such that $X\x\N^+$ embeds in $\R^n$,
then the deleted product $\widetilde{X\x I}$ admits a $\Z/2$-map to $S^{n-1}$.
\end{lemma}

Here $\Z/2$ acts on $\tilde Y:=Y\x Y\but\Delta$ by exchanging the factors and on $S^{n-1}$ by the antipodal involution $x\mapsto -x$.

\begin{proof} 
Let $g=(g_1,g_2,\dots,g_\infty)\:X\x\N^+\emb\R^n$ be an embedding.
Let us endow $X$ with the metric induced by $g_\infty$.
Let $R$ be a $\Z/2$-invariant regular neighborhood of the diagonal in $X\x X$, so that
$X\x X\but R$ is $\Z/2$-homeomorphic to $\tilde X$.
Then there exists an $\eps>0$ such that $R$ contains the $2\eps$-neighborhood of $\Delta$ 
in the $l_\infty$ product metric on $X\x X$.
Then there exists an $i\in\N$ such that $d(g_i,g_\infty)<\eps$.
Let $Y=X\x[0,1]$ and let us define $g\:Y\to\R^n$ by $g(x,t)=tg_i(x)+(1-t)g_\infty(x)$.
Let $\tilde Y_R=\{\big((x,s),(y,t)\big)\in Y\x Y\mid (x,y)\notin R\text{ or } (t,s)\in\widetilde{\{0,1\}}\}$.
Since $(x,y)\notin R$ implies $d(x,y)\ge 2\eps$, we have $g(p)\ne g(q)$ for any $(p,q)\in\tilde Y_R$.
Hence we may define an equivariant map $\tilde g\:\tilde Y_R\to S^{n-1}$ by $\tilde g(p,q)=\frac{||g(q)-g(p)||}{g(q)-g(p)}$.
Clearly, $\tilde Y_R$ is a $\Z/2$-deformation retract of $\tilde Y$.
Thus we obtain an equivariant map $\tilde Y\to S^{n-1}$.
\end{proof}

\begin{proof}[Proof of Theorem \ref{t1}]
Let us note that $M$ is a pseudo-manifold (possibly with boundary).  
By considering a regular neighborhood in $M$ of an embedded orientation-reversing loop in the dual 1-skeleton of $M$ we may assume that $M$ 
is homeomorphic to the total space of the nonorientable $(n-2)$-disc bundle over $S^1$.
By Lemma \ref{l1} it suffices show that if $N$ is a non-orientable smooth $n$-manifold (namely, $N=M\x I$),
then $\tilde N$ admits no equivariant map to $S^{n-1}$.
By considering the interior of $N$ we may assume that it has no boundary.
Let $SN$ be the total space of the spherical tangent bundle of $N$ and let $t$ be the involution on $SN$ 
that is antipodal on each fiber $S^{n-1}$.
Let $c\in H^1(PN;\,\Z/2)$ be the fundamental class of $t$, i.e.\ the first Stiefel--Whitney class of
the line bundle associated to the double covering $SN\to PN:=SN/t$.
Then according to one of the definitions of the Stiefel--Whitney classes $w_i=w_i(N)$ \cite{CF}*{(6.2)},
$c^n=q^*(w_n)+q^*(w_{n-1})c+\dots+q^*(w_1)c^{n-1}$, where $q\:PN\to N$ is the projectivized tangent bundle of $N$.
Since $N$ is non-orientable, we have $w_1\ne 0$.
Then, since the $w_i$ are uniquely determined by the previous formula (see \cite{CF} and note that the authors are implicitly
using the Leray--Hirsch theorem), we must have $c^n\ne 0$.
But if there exists an equivariant map $\phi\:SN\to S^{n-1}$, then $c=\phi^*(d)$, where $d\in H^1(\R P^{n-1};\,\Z/2)$ 
is the fundamental class of the antipodal involution on $S^{n-1}$, and hence $c^n=\phi^*(d^n)=0$.
\end{proof}

In conclusion, let us note another application of Lemma \ref{l1} (which is used in \cite{Fr}).

\begin{lemma} \label{l2} \cite{SS} If $X$ is a contractible compact polyhedron such that $X\x\N^+$ embeds in $\R^n$,
then the suspension $\Sigma\tilde X$ admits a $\Z/2$-map to $S^{n-1}$.
\end{lemma}

\begin{proof} Since $X$ is contractible, $\Sigma\tilde X$ admits a $\Z/2$-map to the double mapping cylinder of the inclusions 
$X\supset\tilde X\subset X$.
The latter $\Z/2$-embeds in $\widetilde{X\x I}$.
\end{proof}

\begin{theorem} \label{t2} \cite{Yo} $\R^n$ does not contain uncountably many pairwise disjoint copies of the $(n-1)$-dimensional
umbrella $U^{n-1}$, that is, the cone over $S^{n-2}\sqcup pt$.
\end{theorem}

A different proof of Theorem \ref{t2}, also based on Lemma \ref{l2}, is given in \cite{SS}.

\begin{proof}
Let us triangulate $U:=U^{n-1}$ as the cone over $\partial\Delta^{n-1}\sqcup pt$.
Then it is self-dual as a subcomplex of $\Delta^{n+1}$, i.e.\ contains precisely one face out of each pair $\Delta^k$, $\Delta^{n-k}$
of complementary faces of $\Delta^{n+1}$.
Consequently its simplicial deleted join $U\circledast U$ is $\Z/2$-homeomorphic to $S^n$ \cite{M}*{Corollary 3.16}.
Also there exists a $\Z/2$-map from $U\circledast U$ to the suspension over the simplicial deleted product 
$U\otimes U$ (see \cite{M}*{Lemma 3.25}).
Since $U\otimes U\subset\tilde U$, by the Borsuk--Ulam theorem there exists no $\Z/2$-map 
$\Sigma\tilde U\to S^{n-1}$.
\end{proof}

\end{document}